\documentclass[11pt]{article}
\usepackage[round]{natbib}
\usepackage[T1]{fontenc}
\usepackage[utf8]{inputenc}
\usepackage{amsmath}
\usepackage{amsthm}
\usepackage[left=1 in,right=1 in,top=1 in,bottom=1 in]{geometry}
\usepackage{amssymb}
\usepackage{adjustbox}
\usepackage[title]{appendix}
\usepackage{mwe}
\usepackage{bbold}
\usepackage{algorithm}
\usepackage{algpseudocode}
\usepackage{quiver}
\usepackage{textcomp}
\usepackage{siunitx}
\usepackage{xfrac}
\usepackage{bbm}
\usepackage{amsthm}
\usepackage{hyperref}
\usepackage{graphicx}
\usepackage{float}
\usepackage{subcaption}
\usepackage{setspace}
\usepackage{tabto}
\usepackage[nottoc, notlof, notlot]{tocbibind}
\theoremstyle{definition}
\newtheorem{defi}{Definition}
\theoremstyle{plain}
\newtheorem{thm}{Theorem}
\newtheorem{crl}{Corollary}
\newtheorem{lmm}{Lemma}
\theoremstyle{remark}

\theoremstyle{plain}
\newtheorem{prp}{Proposition}
\theoremstyle{plain}

\theoremstyle{plain}

\newtheorem{innercustomthm}{Theorem}
\newenvironment{customthm}[1]
  {\renewcommand\theinnercustomthm{#1}\innercustomthm}
  {\endinnercustomthm}

\allowdisplaybreaks

\usepackage{bigfoot}

\DeclareNewFootnote{AAffil}[arabic]
\DeclareNewFootnote{ANote}[fnsymbol]

\usepackage{etoolbox}
\makeatletter
\patchcmd\maketitle{\def\@makefnmark{\rlap{\@textsuperscript{\normalfont\@thefnmark}}}}{}{}{}
\makeatother

\makeatletter
\def\thanksAAffil#1{
  \footnotemarkAAffil\protected@xdef\@thanks{\@thanks%
        \protect\footnotetextAAffil[\the \c@footnoteAAffil]{#1}}%
}
\def\thanksANote#1{%
  \footnotemarkANote%
  \protected@xdef\@thanks{\@thanks%
        \protect\footnotetextANote[\the \c@footnoteANote]{#1}}%
}
\makeatother
\begin{document}
\title{A Geometric Approach for Multivariate Jumps Detection}
\author{Hugo Henneuse \thanksAAffil{Laboratoire de Mathématiques d'Orsay, Université Paris-Saclay, Orsay, France}$^{\text{ ,}}$\thanksAAffil{DataShape, Inria Saclay, Palaiseau, France}}
\maketitle
\begin{abstract}
Our study addresses the inference of jumps (i.e. sets of discontinuities) within multivariate signals from noisy observations in the non-parametric regression setting. Departing from standard analytical approaches, we propose a new framework, based on geometric control over the set of discontinuities. This allows to consider larger classes of signals, of any dimension, with potentially wild discontinuities (exhibiting, for example, self-intersections and corners). We study a simple estimation procedure relying on histogram differences and show its consistency and near-optimality for the Hausdorff distance over these new classes. Furthermore, exploiting the assumptions on the geometry of jumps, we design procedures to infer consistently the homology groups of the jumps locations and the persistence diagrams from the induced offset filtration.
\end{abstract}
\section*{Introduction}
 Inferring discontinuities from noisy observations of an unknown signal arises in a variety of practical applications and has become a classical subject in non-parametric statistics. A common framework to formalize this problem is to consider the non-parametric regression setting (with fixed regular design), where we observe $n=N^{d}$ points, 
\begin{equation}
    X_{i}=f(x_{i})+\sigma\varepsilon_{i}
    \label{eq: reg model}
\end{equation}
with $(x_{i})_{1\leq i\leq n}$ points of the regular $N^{d}$ grid $G_{n}$ over $[0,1]^{d}$, $\sigma$ the level of noise, $f:[0,1]^{d}\rightarrow \mathbb{R}$ the signal, and where we suppose that $(\varepsilon_{i})_{1\leq i\leq n}$ are independent standard Gaussian variables. In this setting, the main object that we want to estimate is then the set of jumps (or discontinuities) of $f$, i.e.,
$$D_{f}=\left\{x\in[0,1]^{d}\text{ s.t. } \limsup\limits_{y\rightarrow x}f(y)-\liminf\limits_{y\rightarrow x}f(y)> 0\right\}.$$

 A huge body of work has been dedicated to univariate signals, making the assumptions that it admits a unique (or a finite number) of discontinuities to estimate its location and magnitude. Different estimators have been proposed, including using split linear smoother \citep{Mcdonald86, Hall92}, differences between one-sided kernels \citep{Muller92, Qiu92, Wu93, Chu96} or histograms \citep{Gayraud02}, relying on wavelets \citep{Wang95}, among others. \\\\
Motivated by applications to image processing, numerous works have also been dedicated to bivariate cases. Yet, in this case, the discontinuities can have a much more complex structure and thus restrictions are required. Since the development of non-parametric statistics is closely linked with the field of functional approximation, these restrictions generally involve analytic assumptions on the discontinuities.

For instance, \cite{korostelev93} propose a minimax estimation procedure in the Hausdorff distance and the symmetric differences distance, based on piecewise polynomial approximation, for signals of the form $f(x,y)=f_{1}(x,y)1_{y\leq g(x)}+f_{2}(x,y)1_{y> g(x)}$, assuming $f_{1}$, $f_{2}$ are continuous functions and $g$ is a $(L,\beta)-$Holder-continuous function. The set of discontinuity is then given by the image of $g$. Consequently, it does not permit to consider discontinuity sets with multiple points. Under similar restriction on the jumps set, they also show that their approach generalizes to higher dimensional setting.

\cite{Sullivan94} propose a method based on contrast statistics to detect jumps, supposing that the discontinuity sets is a smooth, closed and simple parametrized curve, which is also restrictive (e.g. this does not permit to consider multiple points or corners).

\cite{MULLER1994} propose an estimator based on kernel differences. Their approach is not specific to the bivariate cases and can be applied in any dimension, but suppose that the discontinuities set separate $[0,1]^{d}$ into exactly two connected components, which once again does not allow for multiple points.

\cite{Wang98} offers a  broader framework, allowing for signal with discontinuities sets with multiple points, but supposes knowledge on their numbers. The estimation make use of 2D Daubechies wavelet, and is shown to achieve nearly minimax rates for the Hausdorff distance.

More recently, \cite{Bengs19}, revisit the approach of \cite{MULLER1994} and provide confidence sets, in the bivariate case, with similar assumptions to those described by \cite{korostelev93}.\\\\
All the works previously mentioned consider discontinuities sets as (regular) curves or finite union of (regular) curves (supposing we know their number), and estimates those curves. A different approach, as in \cite{Qiu97, Qiu02, bookQiu, Qiu14}, consider discontinuities as a points set, and estimates it by a points set. This allows to consider much general discontinuities sets. The proposed procedures in these papers are based on kernels differences and are shown to consistently estimates, in Hausdorff distance, the discontinuities outside arbitrary small neighborhoods of "singular points", which typically includes corners and multiple points.

A variant of this approach is proposed by \cite{Muller07}, using $M-$kernel estimators instead of classical kernel estimators, showing better numerical robustness. \\\\ 
Outside the non-parametric statistics sphere, this problem has also received a considerable attention, particularly in the field of computer vision, where edge or contour detection is a central problem. Historical approaches make use of ``Filter operators'' \citep{Roberts65,Prewitt70,Sobel74,Canny86} and can be connected to several works cited previously. Nowadays, competitive methods involve neural network \citep[see e.g. the recent survey by][]{SurveyEdge22}. Although, those procedures outperform significantly others in practice, they suffer from a lack of theoretical understanding. Therefore, we believe it remains valuable to explore simpler approaches where we can provide convergence guarantees with quantified rates.
\subsection*{Contribution}
As highlighted previously, a difficulty to estimate $D_{f}$ is to handle (at least from a theoretical point of view) singular points (e.g. multiple points, corners, cusps, ...). To overcome this issue, we  adopt in this paper a geometric approach. We introduce large classes of piecewise-continuous functions, with control over the geometry of $D_{f}$. We then propose a simple estimation procedure based on histogram difference to estimate $D_{f}$ in Hausdorff distance. We establish that this approach is consistent and achieves nearly minimax rates over the classes we introduced. Furthermore, we show that our framework allows to recover consistently geometrical and topological information about $D_{f}$. More precisely, we derive procedures to infer the homology groups of $D_{f}$ and the persistence diagrams coming from its offset filtration.\\\\
The paper is organized as follows : Section \ref{sec: Framework} describes the formal framework adopted in this paper, Section \ref{sec: jumps} is dedicated to the estimation of $D_{f}$, Section \ref{sec: homology} to the estimation of its homology groups and Section \ref{sec: PH} to the estimation of the persistence diagram from its offsets. Proofs and auxiliary results can be found in the appendix.
\section{Framework}
In this section, we describe the formal framework of this work. We consider the non-parametric regression design, with fixed design defined by (\ref{eq: reg model}). For a set $K\subset[0,1]^{d}$, we denote $\overline{A}$ its adherence and $\partial A$ its boundary. We suppose that $f$ verifies the following assumptions :
\begin{itemize}
    \item \textbf{A1.} $f$ is a piecewise uniformly-continuous function for the continuity modulus $\omega$, i.e. there exist $M_{1},...,M_{l}$ open sets of $[0,1]^{d}$ such that,
$$\bigcup\limits_{i=1}^{l}\overline{M_{i}}=[0,1]^{d}$$ and for all $i\in\{1,...,l\}$ and $x,y\in M_{i}$,
$$|f(x)-f(y)|\leq \omega(||x-y||_{2})\text{ with } \lim\limits_{t\rightarrow 0}\omega(t)=0.$$
\item \textbf{A2.} $f$ verifies, $\forall x_{0}\in [0,1]^{d}$, there exists $i\in\{1,...,l\}$, such that
$$\underset{x\in M_{i}\rightarrow x_{0}}{\liminf}f(x)=f(x_{0}).$$
In this context, two signals, differing only on a null set, are statistically undistinguishable and thus are their discontinuities. Assumption \textbf{A2} prevents such scenario.
\item \textbf{A3.} for $\mu\in]0,1]$ and $R_{\mu}>0$,
$$\operatorname{reach}_{\mu}\left(\bigcup\limits_{i=1}^{l}\partial M_{i}\right)\geq R_{\mu}.$$
This condition involves a common measure in geometric inference, the $\mu$-reach \citep{mureach}. Let a compact set $K$, denote the distance function $d_{K}:x\mapsto \min_{y\in K}||x-y||_{2}$ and $\theta_{K}(x)$ the center of the unique smallest ball enclosing all the closest point from $x$ in $K$. We define the generalized gradient of the distance function by :
$$\nabla_{K}(x)=\frac{x-\theta_{K}(x)}{d_{k}(x)}.$$
The $\mu-$reach of $K$ is then defined by :
$$\operatorname{reach}_\mu(K)=\inf \left\{r \mid \inf_{d_K^{-1}(r)}\left\|\nabla_K\right\|_{2}<\mu\right\}.$$
Although this condition is satisfied for large classes of function (e.g. all piecewise linear compact sets have positive $\mu-$reach, for some $\mu>0$), it is a key in this work, providing sufficient control to infer $D_{f}$ (and its topology) consistently. Notably, this assumption allows $D_{f}$ to display corners and multiple points, cases that were excluded in most works mentioned earlier.
\item \textbf{A4.} there exists $l>0$ such that for all $x\in \bigcup\limits_{i=1}^{l}\partial M_{i}$,
$$\limsup\limits_{y\rightarrow x}f(y)-\liminf\limits_{y\rightarrow x}f(y)> l.$$
It is somehow necessary to suppose a lower bound on the magnitude of discontinuities, as else estimating the discontinuities set is, in general, not possible. Even in the case where $f$ is univariate and continuous except at one point, taking the magnitude of the jump arbitrarily small, $f$ is undistinguishable with high certainty from a continuous signal.
\end{itemize}
The classes of functions verifying \textbf{A1},\textbf{A2},\textbf{A3} and \textbf{A4} is denoted $\mathcal{F}_{d}(\omega,\mu,R_{\mu},l)$.\\\\
To evaluate the performance of our estimator, following the choice of most of the papers cited in the introduction, we consider the Hausdorff distance. More precisely, let $A,B\subset[0,1]^{d}$ two compact sets, the Hausdorff distance between $A$ and $B$ is defined by :
$$d_{H}(A,B)=\max\left(\sup_{a\in A}\inf_{b\in B}||a-b||_{2},\sup_{b\in B}\inf_{a\in A}||a-b||_{2}\right).$$
\label{sec: Framework}
\section{Jumps estimation}
\label{sec: jumps}
We propose a simple estimation procedure for the jumps set based on histogram differences. We suppose that $l$ is known, i.e. knowing some lower bound on the magnitude of the jumps. Let $h>0$, $G_{h}$ the regular grid over $[0,1]^{d}$ of step $h$ and $C_{h}$ the collection of closed hypercubes composing this grid. Consider the histogram estimator of $f$,
$$\hat{f}(x)=\frac{1}{|\{x_{i}\in H(x)\}|}\sum\limits_{x_i\in H(x)}X_{i}$$
with $H(x)$ the hypercube of $C_{h}$ containing $x$. Using this estimator of $f$ we define our estimator of $D_{f}$. Let $r>0$, we introduce the local range function, defined for $x\in[0,1]^{d}$ and a function $f$ by,
$$\operatorname{LR}(f,x)= \limsup\limits_{y\in H(x)^{r}\rightarrow x}f(y)-\liminf\limits_{y\in H(x)^{r}\rightarrow x}f(y).$$
denoting, for a set $A\subset[0,1]^{d}$ and $r>0$, $A^{r}=\{x\in[0,1]^{d}\text{ s.t. }d_{2}(x,A)\leq r\}$. We then estimate $D_{f}$ by,
$$\hat{D}_{f}(r,h)=\left\{x\in[0,1]^{d} \text{ s.t. }\operatorname{LR}(\hat{f},x)\geq l/2\right\}.$$

\begin{figure}[H]
\centering
    \begin{subfigure}{0.35\textwidth}
        \centering
        \includegraphics[scale=0.45]{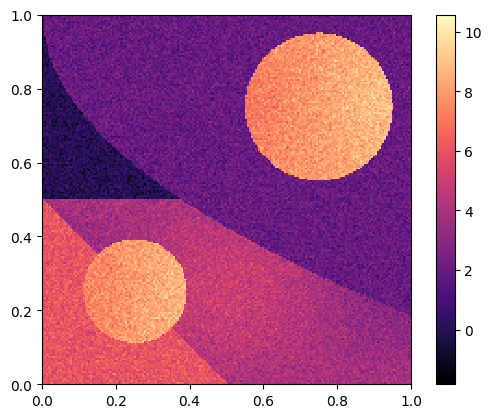}
        \caption{Noisy observation $X$}
    \end{subfigure}
\hfil
    \begin{subfigure}{0.35\textwidth}
    \centering
    \includegraphics[scale=0.45]{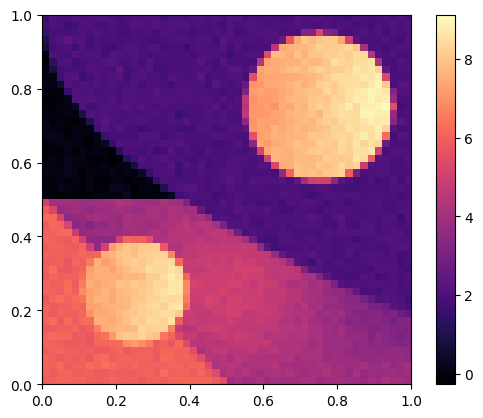}
    \caption{Histogram estimator $\hat{f}$}
\end{subfigure}

    \begin{subfigure}{0.35\textwidth}
        \centering
        \includegraphics[scale=0.45]{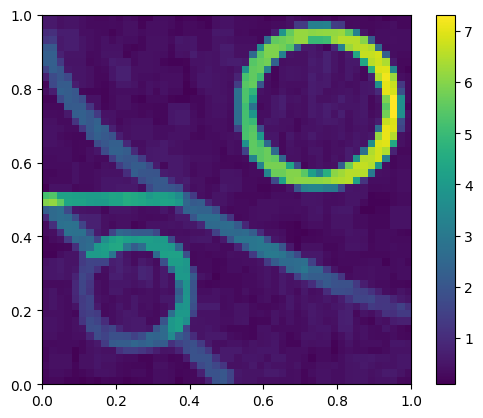}
        \caption{Local range function $\operatorname{LR}(\hat{f},.)$}
    \end{subfigure}
\hfil
    \begin{subfigure}{0.35\textwidth}
    \centering
    \includegraphics[scale=0.44]{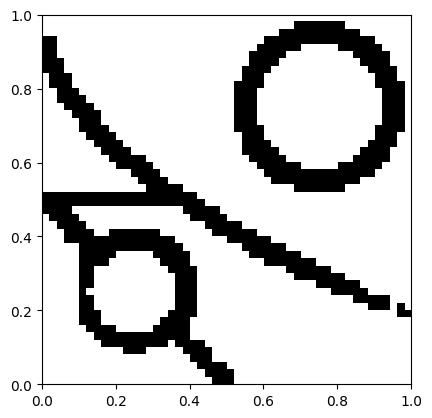}
    \caption{Estimated jumps set $\hat{D}_{f}$}
\end{subfigure}
    \caption{2D numerical illustration of the estimation procedure.}
    \label{Fig: example 1}
\end{figure}
 For a proper choice of $h$ and $r$, our estimator is consistent and achieves the rates provided by Theorem \ref{thm: borne sup}. Interestingly, this rates only depend on $n$ and $d$, the dependence on the parameters $\omega$, $\mu$, $R_{\mu}$ and $l$ only appearing in the constant.
\begin{thm}
\label{thm: borne sup}
Choosing,
$$h_{n}=\begin{cases}
    2\left(512\sigma^{2}/l^{2}\right)^{1/d}\left(\log\left(n^{2}\right)/n\right)^{1/d} \text{ if } \sigma\text{ is known}\\
    s_{n}\left(\log\left(n^{2}\right)/n\right)^{1/d}\text{ with } s_{n}  \text{ a diverging sequence , else.}
\end{cases}$$
and 
$$r_{n}=\begin{cases}
\left(1+\sqrt{d}\right)h_{n}/\mu \text{ if } \mu\text{ is known}\\
s_{n}h_{n} \text{ with } s_{n}  \text{ a diverging sequence, else.}
\end{cases}$$ 
for sufficiently large $n$,
$$\mathbb{P}\left(\sup\limits_{f\in\mathcal{F}_{d}(\omega,\mu,R_{\mu},l)}d_{H}\left(\hat{D}_{f},D_{f}\right)\geq 2r_{n}\right)\leq 2r_{n}.$$
\end{thm}
As $s_{n}$ can be chosen to diverge arbitrary slowly, the obtained convergence rate $r_{n}$ essentially matches the rate obtained in \cite{korostelev93} in the case where $g$ is Lipschitz continuous. From the previous theorem, we can derive a bound in expectation, as stated in Corollary \ref{coro: upper bound}. Both the proof of Theorem \ref{thm: borne sup} and Corollary \ref{coro: upper bound} can be found in Appendix \ref{proof th1}.
\begin{crl}
\label{coro: upper bound}
Choosing $h_{n}$ and $r_{n}$ as in Theorem \ref{thm: borne sup}, we have,
$$\sup\limits_{f\in\mathcal{F}_{d}(\omega,\mu,R_{\mu},l)}\mathbb{E}\left[d_{H}\left(\hat{D}_{f},D_{f}\right)\right]\lesssim r_{n}.$$
\end{crl}
 Now, Theorem \ref{thm: lower bounds} ensures that no estimator can achieve faster rates than $(\log(n)/n)^{1/d}$ and thus proves that our estimation approach is (arbitrarily) nearly minimax. Its proof can be found in Appendix \ref{Proof th2}.
\begin{thm}
\label{thm: lower bounds}
$$\inf\limits_{\hat{D}_{f}}\sup\limits_{f\in\mathcal{F}_{d}(\omega,\mu,R_{\mu},l)}\mathbb{E}\left[d_{H}\left(\hat{D}_{f},D_{f}\right)\right]\gtrsim \left(\log(n)/n\right)^{1/d}.$$
\end{thm}
 Hence, if $\mu$ and $\sigma$ are supposed known, the estimator matches exactly the minimax rates. Otherwise, our procedure is still (arbitrarily) near minimax, for example we can take $s_{n}=\log(n)$ and the procedures achieve minimax rates up to a log factor. Note that the proposed calibration for $h$ and $r$ is chosen to achieve nearly minimax rates asymptotically, relevant practical choice may differ significantly, especially when $n$ is small.
\section{Homology inference}
\label{sec: homology}
One appeal of the geometric framework proposed in this paper, is that it allows to infer interesting geometric and topological information about $D_{f}$. We illustrate this point by proposing a procedure (based on the previous one) to estimate its homology groups. Roughly speaking, the $s-$th ($s\in\mathbb{N}$) homology group, $H_{s}(X)$, of a topological space $X$, describe the (algebraic) structure of the $s-$dimensional ``holes'' of $X$. For an introduction to homology, the reader can refer to chapter 3 of \cite{Hatcher}.\\\\
Denotes $H_{s}$ the (singular) homology functor of degree $s\in\mathbb{N}$, with coefficients in some fixed field $\mathbb{K}$. Due to cubical approximation, $H_{s}(\hat{D}_{f})$ may differ from $H_{s}(D_{f})$. Typically, cycles due to noise and cubical approximation can appear as illustrated in Figure \ref{fig: problem cycle}. We propose a method that permits to remove those problematic cycles, making use of image homology groups, that can be thought as a step of topological regularization.
 Let $\kappa>0$, and,
$$\operatorname{\rho_{s}}:H_{s}\left(\hat{D}_{f}\right)\rightarrow \left(\hat{D}_{f}^{\kappa}\right)$$
the map induced by the inclusion $\hat{D}_{f}\subset \hat{D}_{f}^{\kappa}$. Our estimator for $H_{s}(D_{f})$ is then defined by,
$$\widehat{H_{s}(D_{f})}=\operatorname{Im}\left(\operatorname{\rho_{s}}\right).$$
The following result then establishes the consistency of this procedure. Its proof can be found in Appendix \ref{Proof prp1}.
\begin{prp}
\label{prp: PH1}
Choosing $h_{n}$ and $r_{n}$ as in Theorem \ref{thm: borne sup}, and,
$$\kappa_{n}=\begin{cases}
2r_{n}/\mu^{2} \text{ if } \mu\text{ is known}\\
s_{n}r_{n} \text{ with } s_{n}  \text{ a diverging sequence, else.}
\end{cases}$$ 
we have, for sufficiently large $n$, for all $f\in \mathcal{F}_{d}(\omega,\mu,R_{\mu},l)$ and $s\in\mathbb{N}$,
$$\mathbb{P}\left(H_{s}(D_{f})=\widehat{H_{s}(D_{f})}\right)\geq 1-2r_{n}.$$
\end{prp}
\begin{figure}[H]
    \centering
    \begin{tikzpicture}[x=0.75pt,y=0.75pt,yscale=-1,xscale=1]
\draw  [draw opacity=0][fill={rgb, 255:red, 155; green, 155; blue, 155 }  ,fill opacity=1 ] (369.62,17.3) -- (393.72,17.3) -- (393.72,41.37) -- (369.62,41.37) -- cycle ;
\draw  [draw opacity=0][fill={rgb, 255:red, 155; green, 155; blue, 155 }  ,fill opacity=1 ] (369.62,41.7) -- (393.72,41.7) -- (393.72,65.77) -- (369.62,65.77) -- cycle ;
\draw  [draw opacity=0][fill={rgb, 255:red, 155; green, 155; blue, 155 }  ,fill opacity=1 ] (369.45,66.11) -- (393.55,66.11) -- (393.55,90.17) -- (369.45,90.17) -- cycle ;
\draw  [draw opacity=0][fill={rgb, 255:red, 155; green, 155; blue, 155 }  ,fill opacity=1 ] (345.35,42.04) -- (369.45,42.04) -- (369.45,66.11) -- (345.35,66.11) -- cycle ;
\draw  [draw opacity=0][fill={rgb, 255:red, 155; green, 155; blue, 155 }  ,fill opacity=1 ] (345.35,66.11) -- (369.45,66.11) -- (369.45,90.17) -- (345.35,90.17) -- cycle ;
\draw  [draw opacity=0][fill={rgb, 255:red, 155; green, 155; blue, 155 }  ,fill opacity=1 ] (345.35,90.17) -- (369.45,90.17) -- (369.45,114.24) -- (345.35,114.24) -- cycle ;
\draw  [draw opacity=0][fill={rgb, 255:red, 155; green, 155; blue, 155 }  ,fill opacity=1 ] (321.08,90.34) -- (345.18,90.34) -- (345.18,114.41) -- (321.08,114.41) -- cycle ;
\draw  [draw opacity=0][fill={rgb, 255:red, 155; green, 155; blue, 155 }  ,fill opacity=1 ] (175.47,114.41) -- (199.57,114.41) -- (199.57,138.48) -- (175.47,138.48) -- cycle ;
\draw  [draw opacity=0][fill={rgb, 255:red, 155; green, 155; blue, 155 }  ,fill opacity=1 ] (272.72,90.34) -- (296.82,90.34) -- (296.82,114.41) -- (272.72,114.41) -- cycle ;
\draw  [draw opacity=0][fill={rgb, 255:red, 155; green, 155; blue, 155 }  ,fill opacity=1 ] (321.08,114.41) -- (345.18,114.41) -- (345.18,138.48) -- (321.08,138.48) -- cycle ;
\draw  [draw opacity=0][fill={rgb, 255:red, 155; green, 155; blue, 155 }  ,fill opacity=1 ] (296.82,114.41) -- (320.92,114.41) -- (320.92,138.48) -- (296.82,138.48) -- cycle ;
\draw  [draw opacity=0][fill={rgb, 255:red, 155; green, 155; blue, 155 }  ,fill opacity=1 ] (296.82,138.48) -- (320.92,138.48) -- (320.92,162.54) -- (296.82,162.54) -- cycle ;
\draw  [draw opacity=0][fill={rgb, 255:red, 155; green, 155; blue, 155 }  ,fill opacity=1 ] (272.55,138.48) -- (296.65,138.48) -- (296.65,162.54) -- (272.55,162.54) -- cycle ;
\draw  [draw opacity=0][fill={rgb, 255:red, 155; green, 155; blue, 155 }  ,fill opacity=1 ] (248.28,114.07) -- (272.38,114.07) -- (272.38,138.14) -- (248.28,138.14) -- cycle ;
\draw  [draw opacity=0][fill={rgb, 255:red, 155; green, 155; blue, 155 }  ,fill opacity=1 ] (248.28,90.01) -- (272.38,90.01) -- (272.38,114.07) -- (248.28,114.07) -- cycle ;
\draw  [draw opacity=0][fill={rgb, 255:red, 155; green, 155; blue, 155 }  ,fill opacity=1 ] (248.28,41.54) -- (272.38,41.54) -- (272.38,65.6) -- (248.28,65.6) -- cycle ;
\draw  [draw opacity=0][fill={rgb, 255:red, 155; green, 155; blue, 155 }  ,fill opacity=1 ] (247.77,66.11) -- (271.87,66.11) -- (271.87,90.17) -- (247.77,90.17) -- cycle ;
\draw  [draw opacity=0][fill={rgb, 255:red, 155; green, 155; blue, 155 }  ,fill opacity=1 ] (223.67,42.04) -- (247.77,42.04) -- (247.77,66.11) -- (223.67,66.11) -- cycle ;
\draw  [draw opacity=0][fill={rgb, 255:red, 155; green, 155; blue, 155 }  ,fill opacity=1 ] (199.74,41.7) -- (223.84,41.7) -- (223.84,65.77) -- (199.74,65.77) -- cycle ;
\draw  [draw opacity=0][fill={rgb, 255:red, 155; green, 155; blue, 155 }  ,fill opacity=1 ] (199.57,66.11) -- (223.67,66.11) -- (223.67,90.17) -- (199.57,90.17) -- cycle ;
\draw  [draw opacity=0][fill={rgb, 255:red, 155; green, 155; blue, 155 }  ,fill opacity=1 ] (175.47,65.94) -- (199.57,65.94) -- (199.57,90.01) -- (175.47,90.01) -- cycle ;
\draw  [draw opacity=0][fill={rgb, 255:red, 155; green, 155; blue, 155 }  ,fill opacity=1 ] (151.37,90.17) -- (175.47,90.17) -- (175.47,114.24) -- (151.37,114.24) -- cycle ;
\draw  [draw opacity=0][fill={rgb, 255:red, 155; green, 155; blue, 155 }  ,fill opacity=1 ] (175.47,90.17) -- (199.57,90.17) -- (199.57,114.24) -- (175.47,114.24) -- cycle ;
\draw  [draw opacity=0][fill={rgb, 255:red, 155; green, 155; blue, 155 }  ,fill opacity=1 ] (151.04,114.41) -- (175.14,114.41) -- (175.14,138.48) -- (151.04,138.48) -- cycle ;
\draw  [draw opacity=0] (151.04,17.3) -- (394.4,17.3) -- (394.4,236.6) -- (151.04,236.6) -- cycle ; \draw   (151.04,17.3) -- (151.04,236.6)(175.3,17.3) -- (175.3,236.6)(199.57,17.3) -- (199.57,236.6)(223.84,17.3) -- (223.84,236.6)(248.11,17.3) -- (248.11,236.6)(272.38,17.3) -- (272.38,236.6)(296.65,17.3) -- (296.65,236.6)(320.92,17.3) -- (320.92,236.6)(345.18,17.3) -- (345.18,236.6)(369.45,17.3) -- (369.45,236.6)(393.72,17.3) -- (393.72,236.6) ; \draw   (151.04,17.3) -- (394.4,17.3)(151.04,41.57) -- (394.4,41.57)(151.04,65.84) -- (394.4,65.84)(151.04,90.11) -- (394.4,90.11)(151.04,114.37) -- (394.4,114.37)(151.04,138.64) -- (394.4,138.64)(151.04,162.91) -- (394.4,162.91)(151.04,187.18) -- (394.4,187.18)(151.04,211.45) -- (394.4,211.45)(151.04,235.72) -- (394.4,235.72) ; \draw    ;
\draw [color={rgb, 255:red, 74; green, 144; blue, 226 }  ,draw opacity=1 ][line width=1.5]    (150.4,125.48) .. controls (198.94,92.64) and (249.75,-15) .. (268.03,117.54) .. controls (286.31,250.09) and (373.85,37.14) .. (393.72,17.3) ;
\draw  [color={rgb, 255:red, 208; green, 2; blue, 27 }  ,draw opacity=1 ][line width=1.5]  (272.72,114.41) -- (296.88,114.41) -- (296.88,138.54) -- (272.72,138.54) -- cycle ;
\end{tikzpicture}
    \caption{In blue is represented a true jump sets $D_{f}$ and its estimation $\hat{D}_{f}$ in gray, the false cycle highlighted in red is problematic but can be removed using image homology groups.}
    \label{fig: problem cycle}
\end{figure}
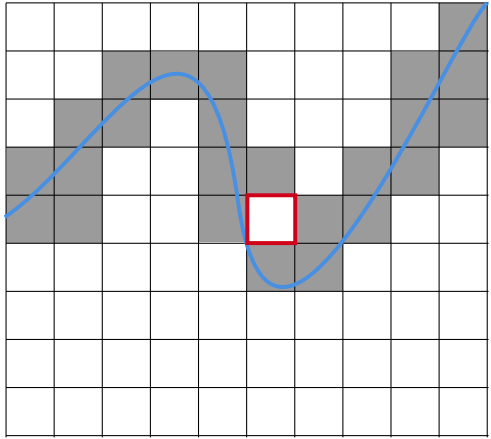
 The cardinal of $H_{s}(D_{f})$, denoted $\beta_{s}(D_{f})$, is the $s-$th Betti number of $D_{f}$, a practical topological descriptor. It simply represents the number of $s-$dimensional holes of $D_{f}$. For instance, in the example illustrated by Figure \ref{Fig: example 1}, the jumps set $D_{f}$ displays two $0-$dimensional holes (i.e. connected components) and two $1-$dimensional holes (i.e. cycles), thus, $\beta_{0}(D_{f})=\beta_{1}(D_{f})=2$ and  $\beta_{s}(D_{f})=0$ otherwise (as $D_{f}\subset[0,1]^{2}$ its homology groups for $s>1$ are trivial).\\
Now, considering, $\widehat{\beta_{s}(D_{f})}$ the cardinal of $\widehat{H_{s}(D_{f})}$, the following corollary is an immediate consequence of Proposition \ref{prp: PH1}.
\begin{crl}
\label{crl: PH1}
Choosing $h_{n}$, $r_{n}$ and $\kappa_{n}$ as in Proposition \ref{prp: PH1}, we have, for sufficiently large $n$, for all $f\in \mathcal{F}_{d}(\omega,\mu,R_{\mu},l)$ and $s\in\mathbb{N}$,
$$\mathbb{P}\left(\beta_{s}(D_{f})=\widehat{\beta_{s}(D_{f})}\right)\geq 1-2r_{n}.$$
\end{crl}
 The computation of $\widehat{\beta_{s}(D_{f})}$ can be achieved using persistent homology and is discussed in the following section.
\section{Persistent homology inference}
\label{sec: PH}
Persistent homology is a prominent concept from Topological Data Analysis that encodes the evolution of topology through a nested family $\mathcal{K}=\left(\mathcal{K}_{\lambda}\right)_{\lambda\in\Lambda}$ (called a filtration), $\Lambda\subset \mathbb{R}$ : birth and death of connected components, cycles, voids,... More precisely, applying the functor $H_{s}$ to $\left(\mathcal{K}_{\lambda}\right)_{\lambda\in\Lambda}$ forms a persistence module $\mathbb{V}_{s}$, which algebraic structure captures the evolution of homology groups of $\mathcal{K}$. Under weak assumptions, the algebraic structure of persistence modules can be characterized by discrete invariants, represented by persistence diagrams. Persistence diagrams (see Figure \ref{fig:diagram}) are collections of points $(b_{i},d_{i})\in \mathbb{R}^{2}$ with $b_{i}$ corresponding to the birth time of a topological feature (e.g. a new connected component appearing at $\lambda=b_{i}$), $d_{i}$ to its death (e.g. the connected component merging with another component born before $b_{i}$ at $\lambda=d_{i}$) and $d_{i}-b_{i}$ to its lifetime. For an overview of persistent homology, with visual illustrations and example of applications, we recommend \cite{chazal2021introduction}. See \cite{chazal2013} for more detailed and rigorous constructions.\\\\
We here propose a procedure to infer consistently the persistence diagram from the nested family $(D_{f}^{\beta})_{\beta\geq 0}$ in bottleneck distance. It provides complementary information to the homology inferred in the previous section, for example, for $s\geq 1$, the lifetime of a class $s-$cycles gives insight on the size of the associated hole. Let $\mathbb{V}_{f,s}$ the persistent module for the $s-$th homology associated to $(D_{f}^{\beta})_{\beta\geq 0}$, and $\operatorname{dgm}(\mathbb{V}_{f,s})$ its persistence diagram. Let also, $\hat{\mathbb{V}}_{f,s}$ the module associated to  $(\hat{D}_{f}^{\beta})_{\beta\geq 0}$ and $\operatorname{dgm}(\hat{\mathbb{V}}_{f,s})$ its persistence diagram.  We prove in Appendix \ref{sec: q-tame} that both these diagrams are well-defined. The following results establish the consistency of our procedure.

Let $D_{1}$ and $D_{2}$ two persistence diagrams and $\Delta=\{(x,x)\in\mathbb{R}^{2}|x>0\}$, the bottleneck distance between $D_{1}$ and $D_{2}$ is then defined by :
$$d_{b}(D_{1},D_{2})=\underset{\gamma\in\Gamma}{\inf}\underset{p\in D_{1}\cup \Delta}{\sup}||p-\gamma(p)||_{\infty}$$
with $\Gamma$ the set of all bijections between $D_{1}\cup \Delta$ and $D_{2}\cup \Delta$. This can be though as a minimal matching distance between the set $D_{1}\cup \Delta$ and $D_{2}\cup \Delta$.
\begin{figure}[H]\centering
    \begin{subfigure}{0.3\textwidth}
    \centering
        \includegraphics[scale=0.5]{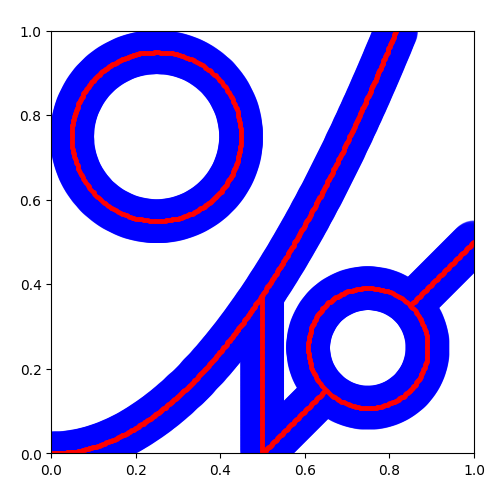}
        \caption{$\beta=0.05$}
    \end{subfigure}
\hfil
    \begin{subfigure}{0.3\textwidth}
    \centering
    \includegraphics[scale=0.5]{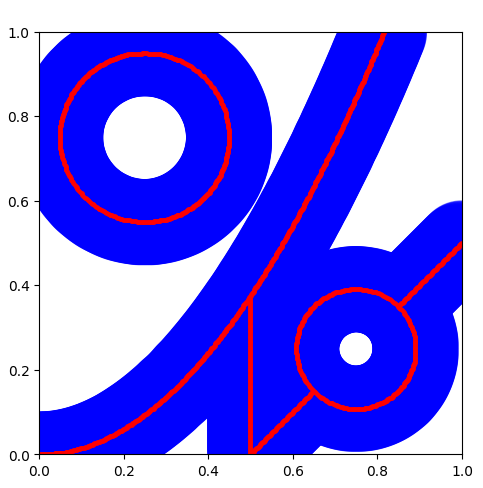}
   \caption{$\beta=0.1$}
\end{subfigure}

    \begin{subfigure}{0.3\textwidth}
    \centering
        \includegraphics[scale=0.5]{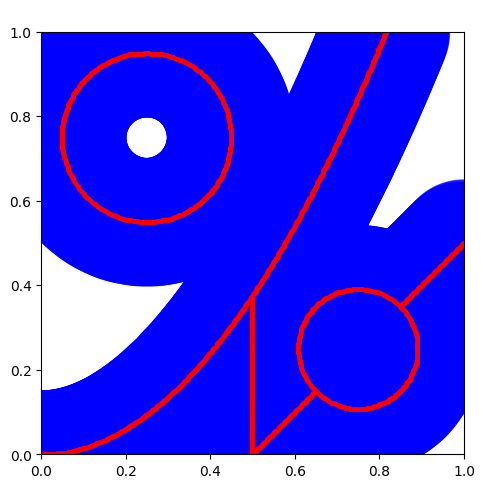}
       \caption{$\beta=0.15$}
    \end{subfigure}
\hfil
    \begin{subfigure}{0.3\textwidth}
    \centering
    \includegraphics[scale=0.5]{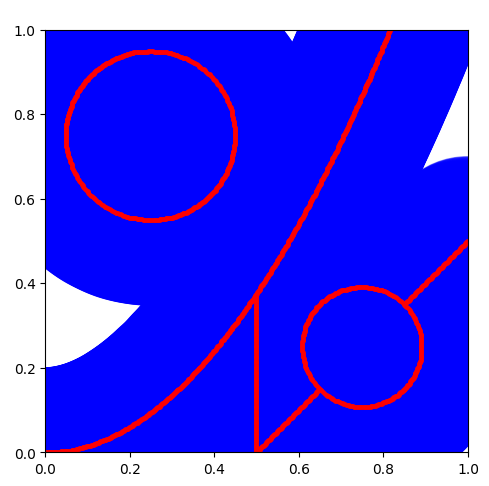}
    \caption{$\beta=0.2$}
\end{subfigure}
    \caption{In red is represented the true jumps set $D_{f}$ and in blue are represented elements of the offset filtration $(D^{\beta}_{f})_{\beta\geq 0}$, for various value of $\beta$, following the example of Figure \ref{Fig: example 1}.}
\end{figure}
\begin{prp}
\label{prp: PH2}
Choosing $h_{n}$ and $r_{n}$ as in Theorem \ref{thm: borne sup}, we have, for sufficiently large $n$, for all $s\in\mathbb{N}$
$$\mathbb{P}\left(\sup\limits_{f\in\mathcal{F}_{d}(\omega,\mu,R_{\mu},l)} d_{b}\left(\operatorname{dgm}(\mathbb{V}_{f,s}),\operatorname{dgm}(\hat{\mathbb{V}}_{f,s})\right)\leq 2r_{n}\right)\geq 1-2r_{n}.$$
\end{prp}

 From this bound, we can deduce a bound in expectation given by the following corollary.
\begin{crl}
\label{coro: PH2}
Choosing $h_{n}$ and $r_{n}$ as in Theorem \ref{thm: borne sup}, we have, for sufficiently large $n$, for all $s\in\mathbb{N}$
$$\mathbb{E}\left(\sup\limits_{f\in\mathcal{F}_{d}(\omega,\mu,R_{\mu},l)} d_{b}\left(\operatorname{dgm}(\mathbb{V}_{f,s}),\operatorname{dgm}(\hat{\mathbb{V}}_{f,s})\right)\right)\lesssim r_{n}.$$
\end{crl}
\textbf{Link with homology inference.} The persistent diagram from the offset filtration $(\hat{D}_{f}^{\beta})_{\beta\geq 0}$ also permits to compute the estimated Betti numbers $\widehat{\beta_{s}(D_{f})}$ introduced in the previous section. It simply suffices to count the cycles classes that has survived between $0$ and $\kappa_{n}$, which corresponds to the number of point (for the $H_{s}-$homology) of $\operatorname{dgm}(\hat{V}_{f,s})$ lying in $\{0\}\times [\kappa_{n},+\infty[$.
\begin{figure}[h]
\centering
\begin{subfigure}{.5\textwidth}
  \centering
  \includegraphics[width=.8\linewidth]{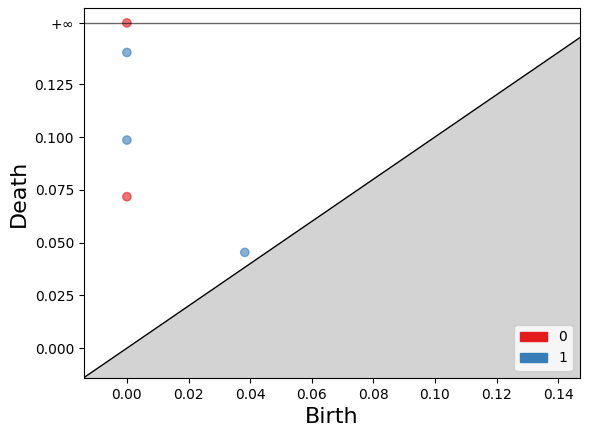}
  \caption{True persistence diagram}
\end{subfigure}%
\begin{subfigure}{.5\textwidth}
  \centering
  \includegraphics[width=.8\linewidth]{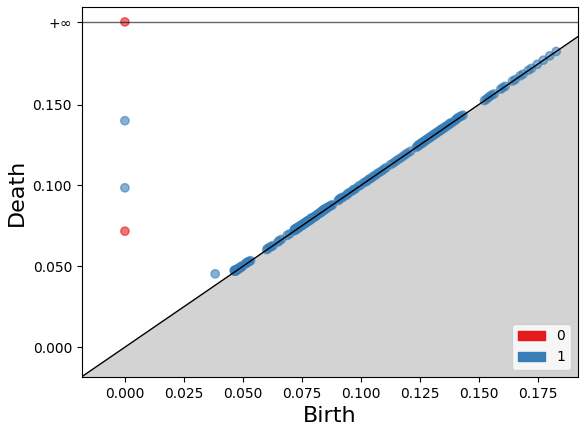}
  \caption{Estimated persistence diagram}
\end{subfigure}
\caption{Persistence diagram for the filtrations $(D_{f}^{\beta})_{\beta\geq 0}$ and $(\hat{D}_{f}^{\beta})_{\beta\geq 0}$, following the example of Figure \ref{Fig: example 1}. In red are represented the $H_{0}-$persistence diagrams and in blue the $H_{1}-$persistence diagrams. Up to noisy features close to the diagonal, topological features from the $(D_{f}^{\beta})_{\beta\geq 0}$ are well approximated.}
\label{fig:diagram}
\end{figure}
\section{Conclusion}
In this paper, we propose a new framework for the inference of the jumps set. It is based on geometric control of the discontinuities and allows corners and multiple points, scenarios often excluded in earlier works. Although the simple histogram differences estimator we introduce is not intended to outperform more sophisticated existing methods in practice, we demonstrate its consistency and near-optimality under the weak assumptions we introduced. This result extends the theoretical scope of jumps inference and, more generally, highlights the advantages of using geometric tools to characterize discontinuities over traditional analytic assumptions.

Additionally, this approach permits the inference of further geometric and topological information about the jumps set, typically, as shown in this work, its homology groups (or Betti numbers) and the persistence diagram from its offsets with convergence guarantees.
\newpage
\bibliographystyle{plainnat}
\bibliography{bibliographie}
\appendix
\section{Proof of Theorem \ref{thm: borne sup} and Corollary \ref{coro: upper bound}}
\label{proof th1}
This section is devoted to the proof of Theorem \ref{thm: borne sup} and Corollary \ref{coro: upper bound}. We start by establishing the following first key lemma.
\begin{lmm}
\label{lmm:th1}
Let $f\in\mathcal{F}_{d}(\omega,\mu,R_{\mu},l)$ and $\overline{C}_{h_{n}}$ the set of hypercubes $H\in C_{h_{n}}$ such that $H\cap \bigcup\limits_{i=1}^{l}\partial M_{i}=\emptyset$. We have,
$$\mathbb{P}\left(\left\|(\hat{f}-f)|_{\overline{C}_{h_{n}}}\right\|_{\infty}> t+\omega\left(\sqrt{d}h_{n}\right)\right)\leq \frac{2}{h_{n}^{d}}\exp(-t^{2}\lfloor h_{n}N\rfloor^{d}/(2\sigma^2)).$$
\end{lmm}
\begin{proof}
Let $H\in \overline{C}_{h_{n}}$ and $x\in H$, By assumption \textbf{A1}, we have,
\begin{align*}
|f(x)-\hat{f}(x)|&=\left|f(x)-\frac{1}{|\{x_{i}\in H\}|}\sum\limits_{x_i\in H}X_{i}\right|\\
&=\left|f(x)-\frac{1}{|\{x_{i}\in H\}|}\sum\limits_{x_i\in H}f(x_{i})-\frac{\sigma}{|\{x_{i}\in H\}|}\sum\limits_{x_i\in H}\varepsilon_{i}\right|\\
&\leq \omega\left(\sqrt{d}h_{n}\right)+\left|\frac{\sigma}{|\{x_{i}\in H\}|}\sum\limits_{x_i\in H}\varepsilon_{i}\right|
\end{align*}
Now, as $\frac{1}{\sqrt{|\{x_{i}\in H\}|}}\sum\limits_{x_i\in H}\varepsilon_{i}$ is a standard Gaussian variable,
\begin{align*}
\mathbb{P}\left(\left|f(x)-\hat{f}(x)\right|> t+\omega\left(\sqrt{d}h_{n}\right)\right)&\leq \mathbb{P}\left(\frac{1}{|\{x_{i}\in H\}|}\sum\limits_{x_i\in H}|\varepsilon_{i}|> t/\sigma\right)\\
&\leq 2\exp\left(-\frac{t^{2}|\{x_{i}\in H\}|^{d}}{2\sigma^{2}}\right)\\
&\leq 2\exp\left(-\frac{t^{2}\lfloor h_{n}N\rfloor^{d}}{2\sigma^{2}}\right).
\end{align*}
Then, by union bound,
$$\mathbb{P}\left(\max\limits_{H\in\overline{C}_{h_{n}}}\left\|\hat{f}_{n}|_{H}-f|_{H}\right\|_{\infty}\geq t+\omega\left(\sqrt{d}h_{n}\right)\right)\leq \frac{2}{h_{n}^{d}}\exp\left(-t^{2}\lfloor h_{n}N\rfloor^{d}/(2\sigma^2)\right).$$
\end{proof}
 Now, we recall a result from \cite{SSDO07}, the second key to establish Theorem \ref{thm: borne sup}.
\begin{lmm}{\cite[Lemma 3.1.]{SSDO07}}
\label{lmm SSDO7}
Let $K\subset [0,1]^{d}$ a compact set and let $\mu>0, r>0$ be such that $r< \operatorname{reach}_\mu(K)$. For any $x \in K^{r} \backslash K$, we have,
$$
d_{2}\left(x, \partial K^{r}\right) \leq \frac{r-d_K(x)}{\mu} \leq \frac{r}{\mu}.
$$
\end{lmm}
 Equipped with this two lemmas, we can now show Theorem \ref{thm: borne sup}.
\begin{proof}[Proof of Theorem \ref{thm: borne sup}]
Let $f\in\mathcal{F}_{d}(\omega,\mu,R_{\mu},l)$ and suppose $n$ sufficiently large such that $r_{n}\geq (1+\sqrt{d})h_{n}/\mu$. Suppose furthermore that,
\begin{equation}
\label{th1:1}
l>8\left\|(\hat{f}-f)|_{\overline{C}_{h_{n}}}\right\|_{\infty}.
\end{equation}
Let $x\in D_{f}$ and $H\in C_{h_{n}}$ the hypercube containing $x$, by Assumption \textbf{A1} and \textbf{A2}, there exist $i,j\in\{1,...,l\}$, $i\ne j$, such that $x\in\partial M_{i}\cap \partial M_{j}$ and for all $z\in M_{i}\cap H^{r_{n}}$,
\begin{equation}
\label{th1:2}
f(z)\geq \liminf\limits_{y\rightarrow x}f(y)+l-\omega(3r_{n})
\end{equation}
and for all $z\in M_{j}\cap H^{r_{n}}$,
\begin{equation}
\label{th1:3}
f(z)\leq \liminf\limits_{y\rightarrow x}f(y)+\omega(3r_{n}).  
\end{equation}
Suppose moreover that $n$ is sufficiently large for $l>16\omega(3r_{n})$ and $3r_{n}<R_{\mu}$. By assumption \textbf{A3}, $\operatorname{reach}_{\mu}\left(M_{i}^{c}\right)\geq R_{\mu}$. Thus, by Lemma \ref{lmm SSDO7}, there exists,
$$y\in \left(\left(M_{i}^{c}\right)^{\left(1+\sqrt{d}\right)h_{n}}\right)^{c} \text{ such that } ||x-y||_{2}\leq r_{n}.$$ 
Denote $\overline{H}\in C_{h_{n}}$ the hypercube containing $y$, we then have $\overline{H}\subset M_{i}\cap H^{r_{n}}$. Similarly, there exists $\underline{H}\in C_{h_{n}}$ such that $\underline{H}\subset M_{j}\cap H^{r_{n}}$. We then have, combining (\ref{th1:2}) and (\ref{th1:3}), under the event (\ref{th1:1}),
\begin{align*}
\hat{f}|_{\overline{H}}-\hat{f}|_{\underline{H}}&\geq \min\limits_{z\in\overline{H}}f(z)-\max\limits_{z\in\underline{H}}f(z)-\|\hat{f}|_{\overline{H}}-f|_{\overline{H}}\|_{\infty}-\|\hat{f}|_{\underline{H}}-f|_{\underline{H}}\|_{\infty}\\
&\geq l-2\omega(3r_{n})-2\left\|(\hat{f}-f)|_{\overline{C}_{h_{n}}}\right\|_{\infty}\\
&> l/2
\end{align*}
Hence, under the event (\ref{th1:1}), $D_{f}\subset \hat{D}_{f}$.\\
Let $x\in \left(D_{f}^{2r_{n}}\right)^{c}\cap M_{i}$, then all hypercubes of $C_{h_{n}}$ intersecting $B_{2}(x,r_{n})$ are contained in $D_{f}^{c}\cap M_{i}$, thus, if (\ref{th1:1}) is verified and, as $l>16\omega(3r_{n})$,
\begin{align*}
\limsup\limits_{y\in H(x)^{r_{n}}\rightarrow x}\hat{f}_{n}(y)-\liminf\limits_{y\in H(x)^{r_{n}}\rightarrow x}\hat{f}_{n}(y)&\leq 2\left\|(\hat{f}-f)|_{\overline{C}_{h_{n}}}\right\|_{\infty}+2\omega(3r_{n})< l/2.
\end{align*}
Thus, under the event (\ref{th1:1}), we also have, $\hat{D}_{f}\subset D_{f}^{2r_{n}}$.\\\\
Remark that by choice of $h_{n}$, for sufficiently large $n$, we have $\lfloor h_{n}N\rfloor^{d}>h_{n}^{d}n/2$. Supposing also $n$ sufficently large for $s_{n}>(1024\sigma^{2}/l^{2})^{1/d}$, by Lemma \ref{lmm:th1}, we then have,
\begin{align*}
\mathbb{P}\left(D_{f}\subset \hat{D}_{f}\subset D_{f}^{2r_{n}}\right)&\geq \mathbb{P}\left(\left\|(\hat{f}-f)|_{\overline{C}_{h_{n}}}\right\|_{\infty}\leq l/8\right)\\
&\geq 1-\frac{2}{h_{n}^{d}}\exp\left(-\frac{\left(l/8-\omega\left(\sqrt{d}h_{n}\right)\right)^{2}}{2\sigma^{2}}\lfloor h_{n}N\rfloor^{d}\right)\\
&\geq 1-\frac{2}{h_{n}^{d}}\exp\left(-\frac{\left(l/8-\omega(3r_{n})\right)^{2}}{2\sigma^{2}}\lfloor h_{n}N\rfloor^{d}\right)\\
&\geq 1- \frac{2}{h_{n}^{d}}\exp\left(-\frac{l^{2}}{1024\sigma^{2}}h_{n}^{d}n\right)\\
&\geq  1- \frac{2}{h_{n}^{d}}\exp\left(-\log(n^{2})\right)\\
&\geq 1-\frac{2}{\log(n^{2})n}\geq 1-2r_{n}
\end{align*}
and the result follows.
\end{proof}
 Corollary \ref{coro: upper bound} then easily follows.
\begin{proof}[Proof of Corollary \ref{coro: upper bound}]
Let denote $E_{n}$ the event $d_{H}\left(\hat{D}_{f},D_{f}\right)\leq 2r_{n}$. By Theorem \ref{thm: borne sup}, we have, for sufficiently large $n$,
\begin{align*}
\mathbb{E}\left[d_{H}\left(\hat{D}_{f},D_{f}\right)\right]&=\mathbb{E}\left[d_{H}\left(\hat{D}_{f},D_{f}\right)1_{E_{n}}\right]+\mathbb{E}\left[d_{H}\left(\hat{D}_{f},D_{f}\right)1_{E_{n}^{c}}\right]\\
&\leq 2r_{n}+\sqrt{d}\mathbb{P}\left(E_{n}^{c}\right)\\
&\leq 2r_{n}+2\sqrt{d}r_{n}\\
&\lesssim r_{n}
\end{align*}
\end{proof}
\section{Proof of Theorem \ref{thm: lower bounds}.}
\label{Proof th2}
This section is devoted to the proof of Theorem \ref{thm: lower bounds}. The proof follows standard methods to provide minimax lower bounds, as presented in section 2 of \cite{TsybakovBook}. The idea is, for any $\delta_{n}=o\left(\left(\log(n)/n\right)^{1/d}\right)$, to exhibit a finite collection of signal in $S_{d}(\omega,\mu,R_{\mu},l)$ such that their discontinuities sets are two by two at distance $2\delta_{n}$ for the Hausdorff distance but indistinguishable, with high certainty.
\begin{proof}[Proof of Theorem \ref{thm: lower bounds}]
Let $l>0$ and consider the half space $D=\{x_{1}\leq 1/2\}$ and,
$$f_{0}(x_{1},...,x_{d})=\begin{cases} 
      0 & \text{ if }  x\in D\\
      l & \text{ else}
\end{cases}.$$
Let $h>0$ and $0<m<1/h$ an integer and $P_{m}$ the (filled) regular pyramid of principal vertex $(1/2+2h,1/2,...,1/2)$ and of principal angle $\theta>\arccos(\mu)/2$, with $\mu>0$ and, 
$$f_{h,m}(x_{1},...,x_{d})=\begin{cases} 
      0 & \text{ if }  x\in D\cup P_{m}\\
      l & \text{ else}
\end{cases}.$$
By construction $f_{0}$ and $f_{h,m}$, for all $0<m<1/h$, are in $\mathcal{F}_{d}(\omega,\mu,R_{\mu},l)$ for all $\omega$ such that $\lim\limits_{t\rightarrow 0}\omega(t)=0$ and some $R_{\mu}>0$ (independent of $h$). Taking $\delta_{n}=h$, we also have, for all $0<m\ne m^{'}<1/h$,
$$d_{H}(D_{f_{0}},D_{f_{h,m}})\geq 2\delta_{n}$$
and 
$$d_{H}(D_{f_{0}},D_{f_{h,m}})\geq 2\delta_{n}.$$
For a fixed signal $f$, denote $\mathbb{P}^{n}_{f}$ the product distribution of the observation $X$. From section 2 of \cite{TsybakovBook}, it now suffices to show that if $\delta_{n}=o\left(\left(\log(n)/n\right)^{1/d}\right)$, then,
\begin{equation}
\label{chi2}
  \frac{1}{\left\lfloor \frac{1}{h}\right\rfloor-2}\sum\limits_{0<m<\lfloor1/h\rfloor}\operatorname{KL}\left(\mathbb{P}^{n}_{f_{h,m}}\mathbb{P}^{n}_{f_{0}}\right)
\end{equation}
with $\operatorname{KL}$ the Kullback-Leibler divergence, converges to zero when $n$ tends to $+\infty$.\\\\ 
Note that for sufficiently large $n$ there exist two constants $c$ and $C$ (possibly depending on $d$ and $\mu$) such that,
$$c\operatorname{Vol}(P_{m})n\leq \left|\{x_{i}\in P_{m}\}\right|\leq C\operatorname{Vol}(P_{m})n .$$
Furthermore, there exists a constant $K$ (depending on $d$ and $\mu$) such that $\operatorname{Vol}(P_{m})=Kh^{d}$. We then have,
\begin{align*}
\mathbb{E}_{\mathbb{P}^{n}_{f_{0}}}\left(\operatorname{KL}\left(\mathbb{P}^{n}_{f_{h,m}}\mathbb{P}^{n}_{f_{0}}\right)\right)&=\prod\limits_{i=1}^{n}\operatorname{KL}\left(\mathcal{N}(f_{0}(x_{i}),\sigma^{2}), \mathcal{N}(f_{h,m}(x_{i}),\sigma^{2})\right)\\
&=\prod\limits_{i=1}^{n} \frac{\left(f_{0}(x_{i})-f_{h,m}(x_{i})\right)^{2}}{2\sigma^{2}}\\
&=\prod\limits_{x_{i}\in P_{m}} \frac{\left(f_{0}(x_{i})-f_{h,m}(x_{i})\right)^{2}}{2\sigma^{2}}\\
&\leq \left(\frac{l^{2}}{2\sigma^{2}}\right)^{\left|\{x_{i}\in P_{m}\}\right|}\\
&\leq  \max\left(\left(\frac{l^{2}}{2\sigma^{2}}\right)^{c\operatorname{Vol}(P_{m})n}, \left(\frac{l^{2}}{2\sigma^{2}}\right)^{C\operatorname{Vol}(P_{m})n}\right)\\
&= \max\left(\left(\frac{l^{2}}{2\sigma^{2}}\right)^{cKh^{d}n}, \left(\frac{l^{2}}{2\sigma^{2}}\right)^{CKh^{d}n}\right)
\end{align*}
Hence, if $h=o\left(\left(\log(n)/n\right)^{1/d}\right)$, we have that \ref{chi2} converges to zero. Consequently, if $\delta_{n}=o\left(\left(\log(n)/n\right)^{1/d}\right)$, then $h=o\left(\left(\log(n)/n\right)^{1/d}\right)$ and we get the conclusion.
\end{proof}
\section{Proof of Proposition \ref{prp: PH1}}
\label{Proof prp1}
This section is dedicated to the proof of Proposition \ref{prp: PH1}. To prove this result, we rely on the following lemma that exploits the geometric control provided by Assumption \textbf{A3}. The first part of the claim is Theorem 12 from \cite{kim2020homotopy} and the second part follows easily from their construction. This result involves the notion of deformation retract that we recall here.
\begin{defi}
A subspace $A$ of $X$ is called a \textbf{deformation retract} of $X$ if there is a continuous $F: X \times [0,1] \rightarrow X$ such that for all $x \in X$ and $a \in A$,
\begin{itemize}
    \item $F(x, 0)=x$ 
    \item $F(x, 1) \in A$
    \item $F(a, 1)=a$.
\end{itemize}
The function $F$ is called a (deformation) \textbf{retraction} from $X$ to $A$.
\end{defi}
 Homotopy, and thus homology, is invariant under deformation retract. Thus, a deformation retraction from $X$ to $A$ induces isomorphism between homology groups. More precisely, for all $t\in[0,1]$, $F(.,t)$ induces a morphism $F(.,t)^{\#}:C_{s}(X)\rightarrow C_{S}(X)$ between $s-$chains of $X$ defined by composing each singular $s$-simplex $\sigma: \Delta^s \rightarrow X$ with $F(.,t)$ to get a singular $s$-simplex $F^{\#}(\sigma,t)=F(.,t)\circ\sigma: \Delta^s \rightarrow X$, then extending $F^{\#}(.,t)$ linearly via $F^{\#}\left(\sum_i n_i \sigma_i,t\right)=\sum_i n_i F^{\#}\left(\sigma_i,t\right)=$ $\sum_i n_i F(.,t)\circ \sigma_i$. The morphism $F^{*}(.,t):H_{s}(X)\rightarrow H_{s}(X)$ defined by $[C]\mapsto [F^{\#}(C,t)]$ can be shown to be an isomorphism for all $t\in[0,1]$ \citep[see][pages 110-113]{Hatcher}. In particular, $F^{*}(.,1):H_{s}(X)\rightarrow H_{s}(A)$ is an isomorphism.
\begin{lmm}{\citep[][Theorem 12]{kim2020homotopy}}
\label{deformation retract}
Let $K\subset[0,1]^{d}$, for all $0\leq r<\operatorname{reach}_{\mu}(K)$, there exists $F:B_{2}(K,r)\times [0,1]\rightarrow B_{2}(K,r)$
a deformation retraction from $B_{2}(K,r)$ onto $K$ verifying $R(x,[0,1])\subset B_{2}(x, 2r/\mu^{2})$.
\end{lmm}
\begin{proof}[Proof of Proposition \ref{prp: PH1}]
In our context, we suppose $\operatorname{reach}_{\mu}\left(D_{f}\right)>R_{\mu}$, hence assuming $n$ sufficiently large, $D_{f}^{2r_{n}+\kappa_{n}}$ retracts by deformation onto $D_{f}$. By Theorem \ref{thm: borne sup}, we know that, for sufficiently large $n$, with probability $1-2r_{n}$,
$$D_{f}\subset \hat{D}_{f}\subset D_{f}^{2r_{n}}$$
and hence,
$$D_{f}^{\kappa_{n}}\subset \hat{D}_{f}^{\kappa_{n}}\subset D_{f}^{2r_{n}+\kappa_{n}}.$$
Let now denote $i:H_{s}(D_{f})\rightarrow \widehat{H_{s}(D_{f})}$ the map induced by the inclusion $D_{f}\subset \hat{D}_{f}$, $j_{0}:\widehat{H_{s}(D_{f})}\rightarrow H_{s}(D_{f}^{2r_{n}+\kappa_{n}})$ the map induced by the inclusion $\hat{D}_{f}^{\kappa_{n}}\subset D_{f}^{2r_{n}+\kappa_{n}}$, $H_{f}^{*}:H_{s}(D_{f}^{2r_{n}+\kappa_{n}})\times [0,1]\rightarrow H_{s}(D_{f}^{2r_{n}+\kappa_{n}})$ the map induced by the deformation retraction $H_{f}$ of $D_{f}^{2r_{n}+\kappa_{n}}$ on $D_{f}$ from Theorem 12 of \cite{kim2020homotopy} and finally $j:\widehat{H_{s}(D_{f})}\rightarrow H_{s}(D_{f})$ given by $j=H_{f}^{*}(.,1)\circ j_{0}$.\\\\
It now suffices to prove that for all $[C]\in \widehat{H_{s}(D_{f})}$, $i(j([C]))=[C]$ and, for all $[C]\in H_{s}(D_{f})$, $j(i([C]))=[C]$.\\\\
The second assertion is direct, as if $C\in C_{s}(D_{f})$, $H^{*}_{f}([C],1)=\left[H_{f}^{\#}(C,1)\right]=[C]$ and $i$ and $j_{0}$ are induced by inclusions.\\\\
For the first, let $[C]\in\widehat{H_{s}(D_{f})}$, we can then suppose that $C\in C_{s}(\hat{D}_{f})$. Then $H_{f}^{*}([C],1)=\left[H^{\#}_{f}(C,1)\right]=[C^{'}]$. Denotes $\overline{H}_{f}$ the restriction of $H_{f}$ to $H_{f}(\hat{D}_{f},[0,1])$. $\overline{H}_{f}$ is a deformation retract from $H_{f}(\hat{D}_{f},[0,1])$ to $D_{f}$. Hence, by homotopy invariance of singular homology,
$$[C]=\left[\overline{H}^{\#}_{f}(C,1)\right]=\left[H^{\#}_{f}(C,1)\right]=[C^{'}]\in H_{s}\left(H_{f}(\hat{D}_{f},[0,1])\right).$$
By Lemma \ref{deformation retract}, for $n$ sufficiently large such that $\kappa_{n}>2r_{n}/\mu^{2}$, $H_{f}(\hat{D}_{f},[0,1])\subset  \hat{D}_{f}^{\kappa_{n}}$ thus $[C]=[C^{'}]\in H_{s}\left(\hat{D}_{f}^{\kappa_{n}}\right)$ and consequently $[C]=[C^{'}]\in \widehat{H_{s}\left(\hat{D}_{f}\right)}$. Then, the assertion follows as $i$ and $j_{0}$ are induced by inclusions and the result is proved.
\end{proof}
\section{Proofs of Proposition \ref{prp: PH2} and Corollary \ref{coro: PH2}}
\label{proof prp2}
This section is dedicated to the proof of  Proposition \ref{prp: PH2} and Corollary \ref{coro: PH2}. Standardly, to establish these results, it suffices to show that $\mathbb{V}_{f,s}$ and $\hat{\mathbb{V}}_{f,s}$ are $2r_{n}-$interleaved, with probability $1-2r_{n}$, and the results then follows from the stability theorem \citep{Chazal2009}.
\begin{defi}
\label{def:interleaving}
Two persistence modules $\mathbb{V}=\left(\mathbb{V}_{\lambda}\right)_{\lambda\in I\subset \mathbb{R}}$ and $\mathbb{W}=\left(\mathbb{W}_{\lambda}\right)_{\lambda\in I\subset \mathbb{R}}$ are said to be \textbf{$\varepsilon$-interleaved} if there exist two families of applications $\phi=\left(\phi_{\lambda}\right)_{\lambda\in I\subset \mathbb{R}}$ and $\psi=\left(\psi_{\lambda}\right)_{\lambda\in I\subset \mathbb{R}}$ where $\phi_{\lambda}:\mathbb{V}_{\lambda}\rightarrow\mathbb{W}_{\lambda+\varepsilon}$, $\psi_{\lambda}:\mathbb{W}_{\lambda}\rightarrow\mathbb{V}_{\lambda+\varepsilon}$, and for all $\lambda<\lambda^{'}$ the following diagrams commutes,
\begin{center}
\begin{tikzcd}
	{\mathbb{V}_{\lambda }} && {\mathbb{V}_{\lambda ^{'}}} & {\mathbb{W}_{\lambda }} && {\mathbb{W}_{\lambda ^{'}}} \\
	{\mathbb{W}_{\lambda +\varepsilon }} && {\mathbb{W}_{\lambda '+\varepsilon }} & {\mathbb{V}_{\lambda +\varepsilon }} && {\mathbb{V}_{\lambda '+\varepsilon }} \\
	{\mathbb{V}_{\lambda }} && {\mathbb{V}_{\lambda+2\varepsilon }} & {\mathbb{W}_{\lambda}} && {\mathbb{W}_{\lambda+2\varepsilon }} \\
	& {\mathbb{W}_{\lambda +\varepsilon }} &&& {\mathbb{V}_{\lambda +\varepsilon }}
	\arrow["{\phi _{\lambda }}"', from=1-1, to=2-1]
	\arrow["{\phi _{\lambda ^{'}}}", from=1-3, to=2-3]
	\arrow["{w_{\lambda +\varepsilon }^{\lambda ^{'} +\varepsilon }}"', from=2-1, to=2-3]
	\arrow["{w_{\lambda }^{\lambda ^{'}}}", from=1-4, to=1-6]
	\arrow["{\psi_{\lambda }}"', from=1-4, to=2-4]
	\arrow["{v_{\lambda +\varepsilon }^{\lambda ^{'} +\varepsilon }}"', from=2-4, to=2-6]
	\arrow["{\psi_{\lambda ^{'}}}", from=1-6, to=2-6]
	\arrow["{v_{\lambda }^{\lambda ^{'}+2\varepsilon}}", from=3-1, to=3-3]
	\arrow["{\phi _{\lambda }}"', from=3-1, to=4-2]
	\arrow["{\psi_{\lambda+\varepsilon}}"', from=4-2, to=3-3]
	\arrow["{\psi_{\lambda }}"', from=3-4, to=4-5]
	\arrow["{v_{\lambda }^{\lambda ^{'}}}", from=1-1, to=1-3]
	\arrow["{w_{\lambda }^{\lambda ^{'}+2\varepsilon}}", from=3-4, to=3-6]
	\arrow["{\phi _{\lambda+\varepsilon}}"', from=4-5, to=3-6]
\end{tikzcd}
\end{center}
\end{defi}
\begin{customthm}{ (\citealt[]["algebraic stability"]{Chazal2009})}
Let $\mathbb{V}$ and $\mathbb{W}$ two $q-$tame persistence modules (see definition in Appendix \ref{sec: q-tame}). If  $\mathbb{V}$ and $\mathbb{W}$  are $\varepsilon-$interleaved, then,
$$d_{b}\left(\operatorname{dgm}(\mathbb{V}),\operatorname{dgm}(\mathbb{W})\right)\leq \varepsilon.$$
\end{customthm}
\begin{proof}[Proof of Proposition \ref{prp: PH2}]
Let $f\in\mathcal{F}_{d}(\omega,\mu,R_{\mu},l)$. By Theorem \ref{thm: borne sup}, we know that, for sufficiently large $n$, with probability $1-2r_{n}$,
$$D_{f}\subset \hat{D}_{f}\subset D_{f}^{2r_{n}}$$
Thus, for all $\beta\geq 0$,
$$D_{f}^{\beta}\subset \hat{D}_{f}^{\beta}\subset D_{f}^{\beta+2r_{n}}$$
Consequently, the persistent modules $\mathbb{V}_{f,s}$ and $\hat{\mathbb{V}}_{f,s}$ are $2r_{n}$ interleaved. Furthermore, we can prove that both those modules are $q-$tame (see Appendix \ref{sec: q-tame}). The stability theorem \cite{Chazal2009} then gives,
$$d_{b}\left(\operatorname{dgm}(\mathbb{V}_{f,s}),\operatorname{dgm}(\hat{\mathbb{V}}_{f,s})\right)\leq 2r_{n}$$
and the result is proved.
\end{proof}
 The bound in expectation stated by Corollary \ref{coro: PH2} then follows by elementary calculations.
\begin{proof}[Proof of Corollary \ref{coro: PH2}]
By proposition \ref{prp: PH1}, we have,
\begin{align*}
&\mathbb{E}\left(\sup\limits_{f\in\mathcal{F}_{d}(\omega,\mu,R_{\mu},l)} d_{b}\left(\operatorname{dgm}(\mathbb{V}_{f,s}),\operatorname{dgm}(\hat{\mathbb{V}}_{f,s})\right)\right)\\
&\leq 2r_{n} \mathbb{P}\left(\sup\limits_{f\in\mathcal{F}_{d}(\omega,\mu,R_{\mu},l)} d_{b}\left(\operatorname{dgm}(\mathbb{V}_{f,s}),\operatorname{dgm}(\hat{\mathbb{V}}_{f,s})\right)\leq 2r_{n}\right)\\
&\quad+\sqrt{d}\mathbb{P}\left(\sup\limits_{f\in\mathcal{F}_{d}(\omega,\mu,R_{\mu},l)} d_{b}\left(\operatorname{dgm}(\mathbb{V}_{f,s}),\operatorname{dgm}(\hat{\mathbb{V}}_{f,s})\right)>2r_{n}\right)\\
&\leq  2r_{n}+ 4\sqrt{d}r_{n}\\
&\leq (2+4\sqrt{d})r_{n}.
\end{align*}
\end{proof}
\section{The diagram $\operatorname{dgm}(\mathbb{V}_{f,s})$ and $\operatorname{dgm}(\hat{\mathbb{V}}_{f,s})$ are well-defined}
\label{sec: q-tame}
This section justifies that both the persistence diagram of $\mathbb{V}_{f,s}$ and $\hat{\mathbb{V}}_{f,s}$ are well-defined. To do so, it suffices to check that those persistence modules are $q-$tame \citep[see][]{chazal2013}, i.e. show that for all $0\leq\beta<\beta^{'}$,
$$v_{\beta}^{\beta^{'}}:H_{s}(D_{f}^{\beta})\rightarrow H_{s}(D_{f}^{\beta^{'}})$$
induced by the inclusion $D_{f}^{\beta}\subset D_{f}^{\beta^{'}}$ and,
$$\hat{v}_{\beta}^{\beta^{'}}:H_{s}(\hat{D}_{f}^{\beta})\rightarrow H_{s}(\hat{D}_{f}^{\beta^{'}})$$
induced by the inclusion $\hat{D}_{f}^{\beta}\subset \hat{D}_{f}^{\beta^{'}}$ are of finite rank.\\\\
Let $0\leq\beta<\beta^{'}$. As $D_{f}^{\beta}$ is a compact set and $[0,1]^{d}$ triangulable, $D_{f}^{\beta}$ is covered by finitely many cells of the triangulation. Hence, there is a finite simplicial complex $K$ such that,
$$D_{f}^{\beta}\subset K\subset D_{f}^{\beta^{'}}.$$
Hence, $v_{\beta}^{\beta^{'}}$ factors through the finite dimensional space $H_{s}(K)$ and is then of finite rank by Theorem 1.1 of \cite{Crawley2012}. And thus, $\mathbb{V}_{f,s}$ is $q-$tame. As $\hat{D}_{f}^{\beta}$ is also compact, the same reasoning gives also the $q-$tameness of $\hat{\mathbb{V}}_{f,s}$.
\section*{Acknowledgements}
The author would like to thank Frédéric Chazal and Pascal Massart for our (many) helpful discussions. The author acknowledge the support of the ANR TopAI chair (ANR–19–CHIA–0001).
\end{document}